\documentclass[final,3p]{elsarticle}

\usepackage{lineno}
\usepackage[colorlinks=true,breaklinks=true,pdftex]{hyperref}
\modulolinenumbers[1]

\usepackage[english]{babel}
\usepackage{listings}
\usepackage{graphicx}
\usepackage{dsfont}
\usepackage{amsmath,amssymb}
\usepackage{xcolor}
\usepackage{cancel}
\usepackage{amsthm}
\usepackage{enumitem}
\usepackage{algpseudocode}
\usepackage{algorithm}

\newcommand{\N}{\mathds{N}}
\newcommand{\R}{\mathds{R}}
\newcommand{\Z}{\mathds{Z}}

\newcommand{\cC}{\mathcal{C}}

\newcommand{\nA}{A}
\newcommand{\nB}{B}

\newcommand{\bd}{d}
\newcommand{\arity}{2}

\theoremstyle{remark}
\newtheorem{thm}{Theorem}
\newtheorem{lem}[thm]{Lemma}

\newtheorem{prop}[thm]{Proposition}

\theoremstyle{remark}
\newtheorem{defi}[thm]{Definition}
\newtheorem{rmk}[thm]{Remark}

\journal{Journal of Computational and Applied Mathematics}

\biboptions{authoryear}

\begin{document}

\begin{frontmatter}
    \title{Activation functions enabling the addition of neurons\\ and layers without altering outcomes}
    \author{Sergio López-Ureña}
    \ead{sergio.lopez-urena@uv.es}
    \cortext[cor]{Corresponding author}
    \address{Dept. de Matem\`atiques, Universitat de Val\`encia, Doctor Moliner Street 50, 46100 Burjassot, Val\`encia, Spain.}
    
    \begin{abstract}
        In this work, we propose activation functions for neuronal networks that are \emph{refinable} and \emph{sum the identity}. This new class of activation functions allows the insertion of new layers between existing ones and/or the increase of neurons in a layer, both without altering the network outputs.
        
        Our approach is grounded in \emph{subdivision} theory. The proposed activation functions are constructed from basic limit functions of convergent subdivision schemes. As a showcase of our results, we introduce a family of spline activation functions and provide comprehensive details for their practical implementation.
    \end{abstract}

    
\end{frontmatter}


\section{Introduction}

In recent years, neural networks (NNs) have achieved remarkable success across various domains, including image recognition, natural language processing, and predictive modeling. This widespread success is largely attributed to their ability to learn complex patterns and representations from vast amounts of data. As a result, the architecture of NNs has evolved to include a large number of parameters and layers, enhancing their capacity to solve intricate problems. However, this growth in complexity also poses significant challenges, particularly in optimizing and managing these extensive networks (as discussed by \cite{DLHSX20}, for instance).

In this line, the optimization of architectures has become a strategic research topic (as explored by \cite{EMH19,PK18}). In these works, the progressive modification of the network and the reuse of fitted parameters are key ideas for efficient training (according to \cite{Lof19}). Growing a network---either by increasing the number of neurons in a layer (widening) or by adding new layers (deepening)---creates opportunities for further optimization. A prominent strategy involves transferring the knowledge acquired by a trained NN to a wider or deeper one through \emph{function-preserving} transformations. Crucially, this approach preserves the function of the original network, ensuring that the new NN produces identical results to its predecessor. Preserving the acquired knowledge facilitates the training process of the new larger NN (as demonstrated in \cite{JSH19}).

Existing methods, such as Net2Net by \cite{CGS15} and NetMorph by \cite{WWRC16}, have shown promising results compared to training enlarged networks from scratch. However, as \cite{WWCW19} observes, these approaches suffer from an initial performance drop during training due to excessive zero-padding in parameter matrices. To this issue, \cite{WWCW19} proposes tensor decomposition with randomized initialization, a computationally demanding solution. For layer insertion, prior methods relied on idempotent activation functions (a very restrictive condition, as in \cite{CGS15}) or parameter-dependent activation functions, consisting in a convex combination of some activation function with the identity, as in \cite{WWRC16,WWCW19}.

The present work introduces a novel framework grounded in approximation theory for NN function-preservation. For layer widening, we propose a subdivision-inspired splitting mechanism wherein each neuron is decomposed into scaled and shifted copies. This allows for more localized adjustments within the network, akin to multiresolution analysis, and is achieved through the use of \emph{refinable} activation functions. For layer insertion, we present parameter-free activation functions that sum the identity---a condition that is less restrictive than idempotence. This solution also emerges naturally from subdivision theory, extending its conceptual foundations to neural network design.

The proposed method offers three key advantages. First, it provides an explicit and efficient formula for computing parameters, which circumvents computationally expensive routines and extensive zero padding. Second, it allows for refined modifications via scaling-shifting mechanisms that align with the principles of multiresolution analysis. Finally, it enables layer insertion, regardless of the activation functions employed in other layers.

From now on, we consider \emph{multilayer perceptrons} defined by sequences of layer operators $L:\R^{n_0}\to\R^{n_1}$,
$$L(x) = \sigma(Wx + b) = \left[\sigma\left(\sum_{j=0}^{n_0-1} w_{i,j}x_j + b_i\right)\right]_{i=0}^{n_{1}-1},$$
where the parameters $W\in\R^{n_1\times n_0}$ and $b\in\R^{n_1}$ are referred to as \emph{weights} and \emph{biases}, respectively. Here, $\sigma:\R\to\R$ denotes the activation function applied componentwise to the vector $Wx+b$.

Our method hinges on two key properties: The activation function must be \emph{refinable} and must \emph{sum the identity}. The former permits the subdivision of a neuron into multiple neurons (Theorem \ref{thm_increase_neurons}), while the latter enables the insertion of new layers (Theorems \ref{thm_insert_layer_pre} and \ref{thm_insert_layer_post}), both operations maintaining the outcomes of the NN unchanged.

\begin{defi} \label{def_refinable}
    A function $\sigma:\R\to\R$ is \emph{refinable} if there exists $\tau\in\R$ and some coefficients $a_l\in\R$, $l=0,\ldots,\nA-1$, $\nA\in\N$, such that
    \begin{equation}\label{eq_refinable}
        \sigma(t) = \sum_{l=0}^{\nA-1} a_l\sigma(\arity t + \tau - l ), \quad \forall t\in\R.
    \end{equation}
\end{defi}

\begin{defi} \label{defi_sum_identity}
    Let $I\subset \R$ be an interval containing $0$ in its interior.
    A function $\sigma:\R\to\R$ \emph{sums the identity} in $I$ iff $\exists \mu\in\R$ and $B\in\N$ such that
    \begin{equation}\label{eq_identity}
        t = \sum_{l=0}^{\nB-1} \sigma(t + \mu - l), \qquad \forall t\in I.
    \end{equation}
\end{defi}

Activation functions constructed from \emph{basic limit functions} of \emph{convergent subdivision schemes} are defined in this paper. For an introduction to this topic, readers are referred to \cite{CDM90, Dyn92, DL02}. The necessary concepts required to understand the contents of this paper are provided herein.

Subdivision theory provides numerous examples of refinable functions, such as the B-Splines, which are examined by \cite{Dyn92}. However, these refinable functions are compactly supported, which inherently makes them non-monotone, whereas activation functions are typically monotone. In Section \ref{sec_subdivision}, we demonstrate how non-decreasing refinable functions can be constructed using compactly supported refinable functions. 

Previous research has connected the theory of refinable functions and subdivision schemes with NN theory, though in a different way from our approach. \cite{Daubechies22} studied the capability of Multi-Layer Perceptrons to approximate refinable functions. That paper and its references illustrate the wide range of applications for refinable functions, including Computer-Aided Design (through subdivision theory), multiresolution analysis (such as wavelet theory), and Markov chains. NNs have also been employed to design subdivision schemes in a data-driven approach, as shown by \cite{LKCAJ20}.

Before addressing the general case in Section \ref{sec_subdivision}, an example of application of our results is provided in Section \ref{sec_bsplines}. In particular, activation functions based on B-Splines are presented, which are refinable and sum the identity. These activation functions, along with their derivatives, are straightforward to compute, making them well-suited for practical applications.

To introduce the results of this work, two activation functions based on B-Splines are now presented:
\begin{equation}\label{eq_sigma:bsplines}
    \sigma_{B^1}(t)
    = \begin{cases}
        -\frac12, & \text{if }t\leq -\frac12,\\[5pt]
        t, & \text{if }-\frac12\leq t\leq \frac12,\\[5pt]
        \frac12, & \text{if }\frac12 \leq t,
    \end{cases}
    \qquad
    \sigma_{B^2}(t) = \begin{cases}
        -\frac12, & \text{if }t\leq -1,\\[5pt]
        t\left(1-\frac{|t|}2\right), & \text{if }-1\leq t\leq 1,\\[5pt]
        \frac12,  & \text{if }1\leq t.
    \end{cases}
\end{equation}
The graphs of these functions can be found in Figure \ref{fig:refinement}. It can be seen that $\sigma_{B^1}$ is a piecewise linear function, which is a shifted version of a clipped ReLU activation function.

The activation functions $\sigma_{B^1},\sigma_{B^2}$ are refinable, since
\begin{equation} \label{eq_introduction_refinable}
    \sigma_{B^1}(t) = \frac12\sigma_{B^1}\left(2t+\frac12\right) + \frac12\sigma_{B^1}\left(2t-\frac12\right), \quad
    \sigma_{B^2}(t) = \frac14\sigma_{B^2}(2t+1) + \frac12\sigma_{B^2}(2t) + \frac14\sigma_{B^2}(2t-1).
\end{equation}
We provide a Wolfram Mathematica notebook to facilitate the verification of some computations discussed in this paper, including those in equations \eqref{eq_introduction_refinable} and \eqref{eq_sigmaB2_sums_identity}. Details on accessing the notebook can be found in the Reproducibility section.
From \eqref{eq_introduction_refinable} we can see that $\sigma_{B^1}$ fulfils Definition \ref{def_refinable} with $\tau = \frac12$, $\nA=2$, $a_0=\frac12$, $a_1=\frac12$, and $\sigma_{B^2}$ with $\tau = 1$, $\nA=3$, $a_0=\frac14$, $a_1=\frac12$, $a_2=\frac14$.
Another example of refinable activation functions is the linear one, $\sigma(x)=\text{id}(x)=x$, fulling the refinability property for any $\nA\in\N$ with $\tau = (\nA-1)/2$, $a_l=\frac1{2 A}$, $l=0,\ldots,\nA-1$. See Figure \ref{fig:refinement} for an illustration of these activations functions and their refinability properties.

\begin{figure}[h] \label{fig:refinement}
    \centering
    \includegraphics[width=0.32\textwidth]{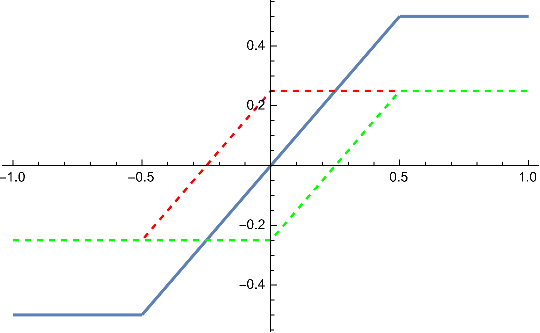}
    \includegraphics[width=0.32\textwidth]{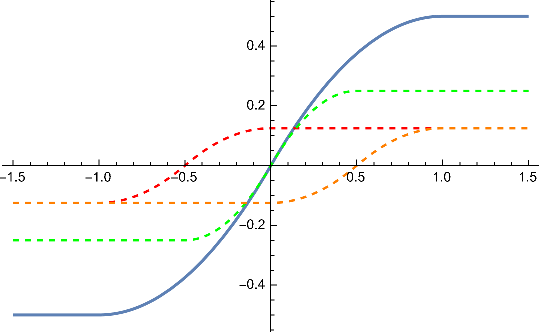}
    \includegraphics[width=0.32\textwidth]{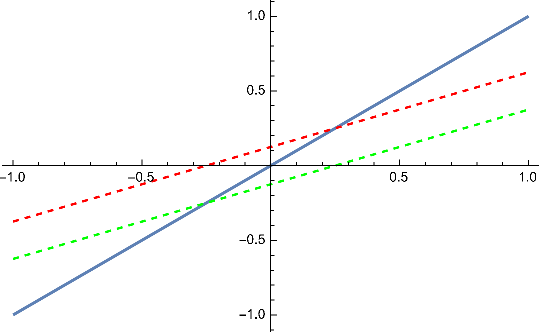}
    \caption{
    An illustration of \eqref{eq_introduction_refinable} is presented. The continuous blue lines represent the refinable activation functions $\sigma_{B^1},\sigma_{B^2},\text{id}$ (from left to right), while the dashed lines show these same functions scaled by 2, shifted and multiplied by a constant. In each graph, the dashed lines sum the continuous blue line.
    }
\end{figure}

These activation functions also sum the identity, as asserted in Theorem \ref{thm_sigmaB_properties}. It is evident that $\text{id}(t) = t$ and $\sigma_{B^1}(t) = t$, $t\in[-\frac12,\frac12]$, while it can be checked that
\begin{equation} \label{eq_sigmaB2_sums_identity}
    \sigma_{B^2}\left(t + \frac12\right) + \sigma_{B^2}\left(t - \frac12\right) = t, \qquad t\in\left[-\frac12,\frac12\right].
\end{equation}  

In general, for every convergent subdivision scheme that generates first-degree polynomials (commonly met properties), the activation function proposed in Definition \ref{defi_sigma_from_phi} is refinable and sums to the identity, as demonstrated in Theorem \ref{thm_sigma_properties}.

The paper is structured as follows. Section \ref{sec_refining} demonstrates the process of inserting new layers and adding neurons to existing ones without altering the outcomes. Pseudocodes for such function-preserving transformations are provided in \ref{algorithms}. Section \ref{sec_bsplines} introduces the spline activation functions and their properties. Section \ref{sec_subdivision} presents more general results, allowing the construction of new activation functions using subdivision theory. 
Finally, Section \ref{sec_conclusions} offers concluding remarks and outlines future research directions.

\section{Refining a neural network} \label{sec_refining}
\subsection{Increasing the number of neurons of a layer by subdivision} \label{sec_increase_neurons}

This section demonstrates that each neuron can be split into $\nA$ neurons without modifying the outcomes of the NN, where $\nA$ was introduced in Definition \ref{def_refinable}. Consequently, the number of neurons in any layer can be multiplied by $\nA$. To achieve this, the weights and biases must be updated according to the formulas provided in Theorem \ref{thm_increase_neurons}.

Consider a NN that includes three consecutive layers of neurons, comprising $n_0,n_1,n_2$ neurons, respectively. We denote the operators connecting them by
\begin{align*}
    L^0 &: \R^{n_0} \to \R^{n_1}, \quad L^0(x) = \sigma^0(W^0 x + b^0), \quad W^0\in\R^{n_1\times n_0}, \quad b^0\in\R^{n_1},\\
    L^1 &: \R^{n_1} \to \R^{n_2}, \quad L^1(x) = \sigma^1(W^1 x + b^1), \quad W^1\in\R^{n_2\times n_1}, \quad b^1\in\R^{n_2},
\end{align*}
where $\sigma^0$ is a refinable activation function, while $\sigma^1$ may not be refinable. We propose to split a neuron of the middle layer by updating $L_0,L_1$ with the two new layer operators
\begin{align*}
    \overline{L}^0 &: \R^{n_0} \to \R^{\bar n_1}, \quad \overline{L}^0(x) = \sigma^0(\overline{W}^0 x + \overline{b}^0), \quad \overline{W}^0\in\R^{\bar n_1\times n_0}, \quad \overline{b}^0\in\R^{\bar n_1},\\
    \overline{L}^1 &: \R^{\bar n_1} \to \R^{n_2}, \quad \overline{L}^1(x) = \sigma^1(\overline{W}^1 x + b^1), \quad \overline{W}^1\in\R^{n_2\times \bar n_1},
\end{align*}
in a way that the outcomes of the layers are not altered, or mathematically speaking:
$$L^1\circ L^0 = \overline{L}^1\circ \overline{L}^0.$$
The layer with $n_1$ neurons is increased to $\bar{n}_1 := n_1 + \nA - 1$ neurons. The parameters $\overline{W}^0,\overline{b}^0,\overline{W}^1$ are defined in Theorem \ref{thm_increase_neurons}.

Since the neurons in a layer are arbitrarily sorted, we can focus on subdividing the first neuron of the layer to simplify the notation.
\begin{thm} \label{thm_increase_neurons}
    Let $\sigma^0, \sigma^1$ be activation functions, being $\sigma^0$ refinable as in \eqref{eq_refinable}.
    Given $W^{0}\in \R^{n_1\times n_{0}}$, $W^1\in\R^{n_{2}\times n_{1}}$ and $b^{0}\in\R^{n_{1}}$, we define the new weights $\overline{W}^{0}\in \R^{(n_1 + \nA-1)\times n_{0}}$, $\overline{W}^{1}\in \R^{n_{2}\times (n_1 + \nA-1)}$ and biases $\overline{b}^{0}\in\R^{n_{1}+\nA-1}$ as
    \begin{alignat*}{4}
        &\overline{W}^{0}_{l,:} := \arity W^{0}_{0,:}, && \quad \overline{b}^{0}_l := \arity b^{0}_0 +\tau - l, && \quad \overline{W}^{1}_{:,l} := a_l W^{1}_{:,0}, & \qquad l= 0, 1, \ldots, \nA-1,\\
        &\overline{W}^{0}_{i+\nA-1,:} := W^{0}_{i,:}, && \quad \overline{b}^{0}_{i+\nA-1} := b^{0}_i, && \quad \overline{W}^{1}_{:,i+\nA-1} := W^{1}_{:,i}, & \qquad i= 1, 2, \ldots, n_1-1,
    \end{alignat*}
    where $W_{i,:}$ and $W_{:,i}$ denote the $i$-th row and column of a matrix $W$, respectively.
    Then,
\(
    L^1 \circ L^{0} = \overline{L}^{1} \circ \overline{L}^{0}.
\)
\end{thm}
\begin{proof}
    First, we show that a part of $L^{0}$ is independent of the zero neuron (the one being subdivided), and that the refinability property \eqref{eq_refinable} can be used in $[L^{0}(x)]_0$:
    \begin{align*}
        [L^{0}(x)]_i &= \sigma^0\left(W^{0}_{i,:}x + b^{0}_i\right) = \sigma^0\left(\overline{W}^{0}_{i+\nA-1,:}x + \overline{b}^{0}_{i+\nA-1}\right) = [\overline{L}^{0}(x)]_{i+\nA-1}, \qquad i= 1, 2, \ldots, n_1-1, \\
        [L^{0}(x)]_0 &= \sigma^0\left(W^{0}_{0,:}x + b^{0}_0\right)
        \overset{\eqref{eq_refinable}}{=} \sum_{l=0}^{\nA-1} a_l \sigma^0(\arity W^{0}_{0,:}x + \arity b^{0}_0 - l + \tau) 
        = \sum_{l=0}^{\nA-1} a_l \sigma^0(\overline{W}^{0}_{l,:}x + \overline{b}^{0}_l)
        = \sum_{l=0}^{\nA-1} a_l \cdot [\overline{L}^{0}(x)]_l,\\
        \Rightarrow & \quad W^1_{:,0} \cdot [L^{0}(x)]_0 = \sum_{l=0}^{\nA-1} a_l W^1_{:,0} \cdot [\overline{L}^{0}(x)]_l = \sum_{l=0}^{\nA-1} \overline{W}^1_{:,l} \cdot  [\overline{L}^{0}(x)]_l.
    \end{align*}
    Now, we conveniently separate the zero neuron (that is being split) from the rest in the definition of $L^1$:
    \begin{align*}
        L^1(x) &= \sigma^1\left(\sum_{i=1}^{n_{1}-1} W^1_{:,i} x_i + W^1_{:,0} \cdot x_0 + b^1\right), \\
        \Rightarrow L^1(L^{0}(x)) &= \sigma^1\left(\sum_{i=1}^{n_{1}-1} W^1_{:,i} \cdot [L^{0}(x)]_i + W^1_{:,0} \cdot [L^{0}(x)]_0 + b^1\right) \\
        &= \sigma^1\left(\sum_{i=1}^{n_1-1} \overline{W}^1_{:,i+\nA-1} \cdot [\overline{L}^{0}(x)]_{i+\nA-1} + \sum_{l=0}^{\nA-1} \overline{W}^1_{:,l} \cdot [\overline{L}^{0}(x)]_l + b^1\right)\\
        &= \sigma^1\left(\sum_{i=\nA}^{n_1+\nA-2} \overline{W}^1_{:,i} \cdot [\overline{L}^{0}(x)]_{i} + \sum_{l=0}^{\nA-1} \overline{W}^1_{:,l} \cdot [\overline{L}^{0}(x)]_l + b^1\right)\\
        &= \sigma^1\left( \sum_{i=0}^{n_1+\nA-2} \overline{W}^1_{:,i} \cdot [\overline{L}^{0}(x)]_{i} + b^1\right) = \overline{L}^1(\overline{L}^{0}(x)).
    \end{align*}
\end{proof}

\begin{rmk}
    Theorem \ref{thm_increase_neurons} can be applied to every neuron, resulting in the multiplication of the number of neurons in the layer by $\nA$. In that case, the weights and biases are updated as follows:
    \begin{alignat*}{4}
        &\overline{W}^{0}_{\nA i+l,:} = \arity \overline{W}^{0}_{i,:}, \quad && \overline{b}^{0}_{\nA i+l} = \arity b^{0}_i+ \tau - l , \quad &&
        \overline{W}^1_{:,\nA i+l} = a_{l} \overline{W}^1_{:,i}, \quad & i=0,\ldots,n_1-1,\quad l=0,\ldots,\nA -1.
    \end{alignat*}
    Observe that this operation can be repeated as many times as desired, arbitrarily increasing the number of neurons in the layer.
\end{rmk}

Algorithm \ref{alg_grow_layer} in \ref{algorithms} shows how Theorem \ref{thm_increase_neurons} can be used to widen a layer while preserving the NN function.

\subsection{Inserting new layers by summing the identity} \label{sec_insert_layer}

In this section, it is demonstrated that new layers can be inserted into a neural network without altering the outcomes, provided that an activation function that sums the identity is used.

Let us consider a NN with two consecutive layers of neurons with $n_0,n_1$ neurons. Let $L$ be the layer operator connecting them:
$$L : \R^{n_0} \to \R^{n_1}, \quad L(x) = \sigma(Wx + b), \quad W\in\R^{n_1\times n_0}, \quad b\in\R^{n_1},$$
where $\sigma$ is an activation function that may not sum the identity.
We are interested in inserting a new layer with $\bar{n}$ neurons between them. 
This implies splitting the layer operator $L$ into two separate operators:
\begin{align*}
    L^0 &: \R^{n_0} \to \R^{\bar n}, \quad L^0(x) = \sigma^0(W^0 x + b^0), \quad W^0\in\R^{\bar n\times n_0}, \quad b^0\in\R^{\bar n},\\
    L^1 &: \R^{\bar n} \to \R^{n_1}, \quad L^1(x) = \sigma(W^1 x + b^1), \quad W^1\in\R^{n_1\times \bar n}, \quad b^1\in\R^{n_1},
\end{align*}
where $\sigma^0$ must sum the identity. Choosing $W^0,W^1,b^0,b^1$ according to Theorems \ref{thm_insert_layer_pre} or \ref{thm_insert_layer_post}, we can ensure that
\begin{equation}\label{eq_insert_layer}
    L(x) = (L_1 \circ L_0)(x), \quad \forall x\in \Omega \subset \R^{n_0},
\end{equation}
for some set $\Omega$, which depends on the interval $I$ where the identity is summed (see Definition \ref{defi_sum_identity}). $\Omega$ can be made arbitrarily large to ensure that \eqref{eq_insert_layer} holds for all data in the training and test sets, as discussed in Remarks \ref{rmk_choose_beta_1} and \ref{rmk_choose_beta_2}. Consequently, the outcomes of the NN remain unchanged for the given data and their surroundings.

The effect of this separation is the insertion of a new layer with $\bar{n}$ neurons between the layer with $n_0$ neurons and the layer with $n_1$ neurons. This can be accomplished in two ways: one with $\bar{n} = B n_0$ (Theorem \ref{thm_insert_layer_pre}) and the other with $\bar{n} = B n_1$ (Theorem \ref{thm_insert_layer_post}), where $B\in\mathbb{N}$ is determined by the identity summing property (see Definition \ref{defi_sum_identity}). Hence, the first one might be convenient when the number of neurons in the previous layer is smaller than the number of neurons in the next layer ($n_0\leq n_1$), while the second one might be preferable in the opposite case. The new weights and biases are defined according to these theorems.

\begin{thm} \label{thm_insert_layer_pre}
    For any $W\in\R^{n_1\times n_0}$, $b\in\R^{n_1}$, $\beta>0$ and any $\sigma^0$ summing the identity, we have that  $L(x) = (L_1 \circ L_0)(x)$, $\forall x\in \Omega \subset \R^{n_0}$ holds for $\bar n = B n_0$ and
    \begin{align*}
        \Omega &:= \{x\in\R^{n_0} : \beta x_i \in I, \ i = 0,\ldots,n_0-1\}, \\
        W^0_{l+iB,:} &:= \beta (e^i)^T, \quad
        b^0_{l+iB} := \mu - l,\quad
        W^1_{:,l+iB} := \frac1\beta W_{:,i},\quad
        b^1 := b,
    \end{align*}
    for $i = 0,\ldots,n_0-1$, $l = 0,\ldots,B-1$, and $e^i$ is the $i$-th vector of the canonical basis of $\R^{n_0}$.
\end{thm}
\begin{proof}
    Let $x\in\Omega$ be.
    First, we separate the contribution of each neuron in the layer operator $L$:
    \begin{align*}
        L(x) &= \sigma(Wx + b) = \sigma\left(\sum_{i = 0}^{n_0-1} W_{:,i}x_i + b\right).
    \end{align*}
    We use that $\sigma$ sums the identity as follows:
    \begin{align*}
        x_i = \frac1\beta \beta x_i \overset{\eqref{defi_sum_identity}}{=} \frac1\beta \sum_{l=0}^{B-1} \sigma^0(\beta x_i + \mu - l),
    \end{align*}
    which holds because $\beta x_i \in I$ by hypothesis.
    Now, we can write
    \begin{align*}
        L(x) &= \sigma\left(\sum_{i = 0}^{n_0-1} W_{:,i}\frac1\beta \sum_{l=0}^{B-1} \sigma^0(\beta x_i + \mu - l) + b\right)
        = \sigma\left(\sum_{i = 0}^{n_0-1} \sum_{l=0}^{B-1} \beta^{-1}W_{:,i} \sigma^0(\beta x_i + \mu - l) + b\right) \\
        &= \sigma\left(\sum_{i = 0}^{n_0-1} \sum_{l=0}^{B-1} W^1_{:,l+iB} \sigma^0(\beta x_i + \mu - l) + b\right).
    \end{align*}
    Observe that
    \[
    \sigma(\beta x_i + \mu - l) = \sigma(\beta (e^i)^T x + \mu - l) = \sigma(W^0_{l+iB,:}x + b^0_{l+iB}) = [L^0(x)]_{l+iB}.
    \]
    Then,
    \begin{align*}
        L(x) &= \sigma\left(\sum_{i = 0}^{n_0-1} \sum_{l=0}^{B-1} W^1_{:,l+iB} [L^0(x)]_{l+iB} + b\right) = \sigma\left(\sum_{i = 0}^{B n_0-1} W^1_{:,i} [L^0(x)]_i + b\right) = (L^1 \circ L^0)(x).
    \end{align*}
\end{proof}

\begin{rmk} \label{rmk_choose_beta_1}
    If practice, it is desirable that the NN remains invariable under this layer-addition operation for some set of data. If $0$ is in the interior of $I$ (as demanded in Definition \ref{defi_sum_identity}), we can choose $\beta$ sufficiently small to guarantee this: Let be $\delta>0$ such that $(-\delta, \delta)\subset I$. If $L$ was processing some vectors $x$ belonging to some set $D\subset\R^{n_0}$, we can choose
$$\beta := \frac{\delta}{2\sup\{ |y_j| \ : \ j\in\{0,\ldots,n_0-1\}, \ y\in D\}},$$
    since it implies that
$$|\beta x_i| = \frac{\delta |x_i|}{2 \sup\{ |y_j| \ : \ j\in\{0,\ldots,n_0-1\}, \ y\in D\} }< \delta$$
    and then $\beta x_i \in I$ for all $i=0,\ldots,n_0-1$, $x\in D$. That is, $D\subset \Omega$.
\end{rmk}

\begin{thm} \label{thm_insert_layer_post}
    For any $W\in\R^{n_1\times n_0}$, $b\in\R^{n_1}$, $\beta>0$ and any $\sigma^0$ summing the identity, we have that \eqref{eq_insert_layer} holds for $\bar n = B n_1$ and
    \begin{align*}
        \Omega &:= \{x\in\R^{n_0} : \beta(W_{i,:} x + b_i) \in I, \ i = 0,\ldots,n_1-1\}, \\
        W^0_{i+l n_1,:} &:= \beta W_{i,:}, \quad
        b^0_{i+l n_1} := \beta b_i + \mu - l,\quad
        W^1_{i,:} := \beta^{-1} \sum_{l=0}^{\nB-1}(e^{i+l n_1})^T,\quad
        b^1 := 0,
    \end{align*}
    for $i = 0,\ldots,n_1-1$, $l = 0,\ldots,\nB-1$, and $e^i$ is the $i$-th vector of the canonical basis of $\R^{\nB n_1}$.
\end{thm}
\begin{proof}
    Let $x\in \Omega$ and $i\in\{0,\ldots,n_1-1\}$ be. Using that $\beta(W_{i,:}x + b_i)\in I$ and the sum of the identity property, we have that
    \begin{align*}
        [L(x)]_i &= \sigma(W_{i,:}x + b_i) = \sigma(\beta^{-1} \beta(W_{i,:}x + b_i))
        \overset{\eqref{eq_identity}}{=} \sigma\left(\beta^{-1} \sum_{l=0}^{\nB-1} \sigma^0(\beta(W_{i,:}x + b_i) + \mu - l)\right) \\
        &= \sigma\left(\sum_{l=0}^{\nB-1} \beta^{-1} \sigma^0(W^0_{i+l n_1,:} x + b^0_{i+l n_1})\right)
        = \sigma\left(\beta^{-1} \sum_{l=0}^{\nB-1}(e^{i+l n_1})^T \sigma^0(W^0 x + b^0)\right)\\
        &= \sigma\left(W^1_{i,:} L^0(x)\right) = [L^1(L^0(x))]_i.
    \end{align*}
\end{proof}

\begin{rmk} \label{rmk_choose_beta_2}
    As in Remark \ref{rmk_choose_beta_1}, we can choose $\beta$ sufficiently small to guarantee that $D\subset \Omega$, where $D$ is the set of data that was being processed by $L$.
    
    Let $\delta>0$ be such that $(-\delta, \delta)\subset I$. We select
    $$\beta := \frac{\delta}{2\sup\{ |(W_{i,:} y + b_i)| \ : \ i\in\{0,\ldots,n_1-1\}, \ y\in D\}}.$$
    For all $x\in D$, it implies that
    $$|\beta(W_{i,:} x + b_i)| = \frac{\delta |W_{i,:} x + b_i|}{2 \sup\{ |(W_{j,:} y + b_j)| \ : \ j\in\{0,\ldots,B n_1-1\}, \ y\in D\} }< \delta$$
    and then $\beta(W_{i,:} x + b_i) \in I$ for all $i=0,\ldots,B n_1-1$. That is, $D\subset \Omega$.
\end{rmk}

The implementations of two methods that insert a new layer into a NN are provided in Algorithms \ref{alg_insert_layer_pre} and \ref{alg_insert_layer_post} in \ref{algorithms}, which use Theorems \ref{thm_insert_layer_pre} and \ref{thm_insert_layer_post}, respectively.

\section{Refinable activation functions that sum the identity}

In this section, a method for constructing activation functions $\sigma:\R\to\R$ that are non-decreasing, continuous, refinable and sum the identity is shown. Before discussing the general results based on subdivision theory, the particular case of B-Splines is introduced.

\subsection{Spline activation functions} \label{sec_bsplines}

Let us consider the B-Splines basis functions defined recursively by
\begin{align}\label{eq_bsplines}
    \phi_{B^0}(t) &= \begin{cases}
        1 & \text{if }0\leq t< 1,\\
        0 & \text{otherwise},
    \end{cases} \quad
    \phi_{B^d}(t) = (\phi_{B^{d-1}}*\phi_{B^{0}})(t) = \int_{t-1}^t \phi_{B^{d-1}}(s)ds, & d\geq 1,
\end{align}
where $*$ denotes the convolution product. In particular, it can be computed (see the Reproducibility section) that
\begin{equation*}
    \phi_{B^1}(t) = \max\{0,1-|t-1|\}, \quad \phi_{B^2}(t) = \begin{cases}
        t^2/2, & \text{if }0< t\leq 1,\\
        \frac{3}{4}-(t-\frac32)^2, & \text{if }1< t\leq 2,\\
        \frac{1}{2} (3-t)^2 & \text{if }2< t\leq 3,\\
        0 & \text{otherwise}.
    \end{cases}
\end{equation*}
A more direct definition is
\begin{equation}\label{eq_bsplines_direct}
    \phi_{B^d}(t) = 
    \sum_{l=0}^{d+1} \frac{(-1)^{l}}{d!}\binom{d+1}{l} \max\{t-l,0\}^{d}, \quad d\geq 1.
\end{equation}
For a brief overview of B-Splines, refer to \cite{Dyn92}, and for a more comprehensive understanding, consult \cite{DeBoor78}. What is significant for our purposes is that $\phi_{B^{\bd}}$ is refinable: 
\begin{equation}\label{eq_bsplines_refinable}
    \phi_{B^{\bd}}(t) = \sum_{l=0}^{\bd+1} 2^{-\bd}\binom{\bd+1}{l}\phi_{B^\bd}(2t-l).
\end{equation}
Another properties to take into account is that $\phi_{B^{\bd}}$ is $\cC^{\bd-1}$, its support is $(0,\bd+1)$ and (according to IX-(v) in \cite{DeBoor78})
\begin{equation}\label{eq_bsplines_identity}
    1 = \sum_{i\in\Z} \phi_{B^{\bd}}(t-i), \quad t-\frac{\bd+1}2 = \sum_{i\in\Z} i \phi_{B^{\bd}}(t-i).
\end{equation}
The equations in \eqref{eq_bsplines_identity} are related to the \emph{generation of first-degree polynomials} property in subdivision theory and are crucial for the construction of activation functions that sum the identity, as we will see.

For a given refinable function, the associated activation function can be defined as follows.
\begin{defi} \label{defi_sigma_from_phi}
    Let $\phi:\R\to\R$ be a refinable function with support contained in $(0,\bd+1)$. The activation function $\sigma:\R\to\R$ associated to $\phi$ is defined by
    \begin{equation} \label{eq_sigma_from_phi}
        \sigma(t) = -\frac12 + \sum_{m=0}^{\infty} \phi\left(t+\frac{\bd}2-m\right).
    \end{equation}
    For $\phi = \phi_{B^\bd}$, we denote $\sigma = \sigma_{B^\bd}$ and we call it the \emph{spline activation function} of degree $d$.
\end{defi}

Now, the main properties of the spline activation function are enumerated. Some of these properties are inherited from the properties of B-Splines, while others are proven using the more general result presented in Theorem \ref{thm_sigma_properties}. The proof can be found in Section \ref{sec_subdivision}.
\begin{thm} \label{thm_sigmaB_properties}
    For any $\bd\in\N$, the activation function $\sigma_{B^\bd}$ fulfils
    \begin{enumerate}[label=(\alph*)]
        \item \label{thm_sigmaB_properties_support} $\sigma_{B^\bd}(t) = -1/2$, if $t\leq -\bd/2$, $\sigma_{B^\bd}(t) = 1/2$, if $\bd/2\leq t$, and $\sigma_{B^\bd}(t) = -\frac12 + \sum_{m=0}^{d-1} \phi_{B^\bd}(t+\frac{\bd}2-m)$, if $-\bd/2\leq t\leq \bd/2$.
        \item \label{thm_sigmaB_properties_smooth}\label{thm_sigmaB_properties_symmetric} It is $\cC^{\bd-1}$ and odd-symmetric.
        \item \label{thm_sigmaB_properties_refinable} It is refinable with $\tau = d/2$, $\nA = \bd+1$ and $a_l = 2^{-\bd}\binom{\bd}{l}$.
        \item \label{thm_sigmaB_properties_sum} For any $B\geq d$, it sums the identity in the interval $I = [-\frac{B-d+1}2,\frac{B-d+1}2]$ with $\mu = \frac{B-1}2$.
        \item \label{thm_sigmaB_properties_expression} $\sigma_{B^\bd}(t) = -\frac12 + \frac1{\bd!}\sum_{l=0}^{\bd} (-1)^l \binom{\bd}{l} \max\{t+d/2-l,0\}^\bd.$
        \item \label{thm_sigmaB_properties_recursive} It fulfils the recursive formula
        $\sigma_{B^{\bd+1}}(t) = \int_{t-1/2}^{t+1/2}\sigma_{B^{\bd}}(s) ds = (U*\sigma_{B^{\bd}})(t),$ where $U$ is the box function: $U(t) = 1$ if $|t|< 1/2$ and $U(t) = 0$ otherwise.
        \item \label{thm_sigmaB_properties_derivative} $\sigma_{B^\bd}'(t) = \phi_{B^{\bd-1}}(t+\frac{d}2) \geq 0$.
    \end{enumerate}
\end{thm}

According to the last theorem, the derivative of $\sigma_{B^\bd}$ is known, which is crucial for applying gradient-based optimization algorithms during the training stage. However, to implement backpropagation, it is necessary to compute the derivative with respect to the function value. The following proposition deals with it for the cases $n=1,2$.
\begin{prop} \label{prop_backpropagation}
    \begin{align*}
        \sigma_{B^1}'(t) &= \begin{cases}
            1, & \text{if } |\sigma_{B^1}(t)| < \frac12,\\
            0, & \text{otherwise},
        \end{cases} \quad 
        \sigma_{B^2}'(t) = \begin{cases}
            \sqrt{1+2\sigma_{B^2}(t)}, & \text{if }-\frac12< \sigma_{B^2}(t)\leq 0,\\
            \sqrt{1-2\sigma_{B^2}(t)}, & \text{if }0< \sigma_{B^2}(t)< \frac12,\\
            0, & \text{otherwise}.
        \end{cases}
    \end{align*}
    Equivalently, $\sigma_{B^1}'(t) = U(\sigma_{B^1}(t))$ and $\sigma_{B^2}'(t) = \sqrt{\max\{0,1-2|\sigma_{B^2}(t)|\}}$.
\end{prop}
\begin{proof}
    For any $d\in\N$, by Theorem \ref{thm_sigmaB_properties}-\ref{thm_sigmaB_properties_support}, we know that $\sigma_{B^d}'(t) = 0$ if $|t|\geq\frac{d}2$, that is, if $|\sigma_{B^d}(t)| = \frac12$.  For $d=1$, it is important to note that the function is not differentiable at $t=-1/2$ and $t=1/2$, but the derivative can be computed in the rest of the domain. $\sigma_{B^1}'(t) = 1$, if $|\sigma_{B^1}(t)| < \frac12$, follows from \eqref{eq_sigma:bsplines}.
    
    For $d\geq 2$, the function is differentiable everywhere, but the result is not so direct. By Theorem \ref{thm_sigmaB_properties}-\ref{thm_sigmaB_properties_derivative}, we know that $\sigma_{B^d}'(t)>0$ for $t\in(-\frac{d}2,\frac{d}2)$, thus $\sigma_{B^d}:(-\frac{d}2,\frac{d}2)\to(-\frac12,\frac12)$ is bijective. Using the inverse function theorem (for univariate functions), we know that
    \[
    \sigma_{B^d}'(t) = \frac{1}{(\sigma_{B^d}^{-1})'(\sigma_{B^d}(t))}, \quad |t|<\frac{d}2.
    \]
    
    However, it is not easy to compute the inverse, in general, since it involve solving a polynomial equation of degree $d$. We can compute it for $d=2$: By \eqref{eq_sigma:bsplines},
    \[
    \sigma_{B^2}(t) = \begin{cases}
        t(1+t/2), & \text{if } -1< t\leq 0,\\
        t(1-t/2), & \text{if } 0< t< 1.
    \end{cases}
    \]
    For $t\in(-1,0]$, we can solve $z = t(1 + t/2)\in(-\frac12,0]$, whose solutions are $t = -1\pm\sqrt{1+2 z}$. Since $\sqrt{1+2 z}\in(0,1]$, then $-1+\sqrt{1+2 z}\in(-1,0]$ and $-1-\sqrt{1+2 z}\in[-2,-1)$. In conclusion, the sign must be positive and $\sigma_{B^2}^{-1}(z) = -1+\sqrt{1+2 z}$ for $z\in(-\frac12,0]$. Analogously, we can deduce that $\sigma_{B^2}^{-1}(z) = 1-\sqrt{1-2z}$ for $z\in[0,\frac12)$. See the Reproducibility section for details. Summarizing, we have that
    \[
    \sigma_{B^2}^{-1}(z) = \begin{cases}
        -1+\sqrt{1+2z}, & \text{if } -\frac12 < z\leq 0,\\
        1-\sqrt{1-2z}, & \text{if } 0< z< \frac12.
    \end{cases}
    \]
    Now we compute the derivative of the inverse function
    \begin{align*}
        (\sigma_{B^2}^{-1})'(z) &= \begin{cases}
            \frac{1}{\sqrt{1+2z}}, & \text{if } -\frac12 < z\leq 0,\\
            \frac{1}{\sqrt{1-2z}}, & \text{if } 0< z< \frac12,
        \end{cases}
    \end{align*}
    and, finally,
    \begin{align*}
        \frac1{(\sigma_{B^2}^{-1})'(z) } &= \sqrt{1-2|z|} = \begin{cases}
            \sqrt{1+2z}, & \text{if } -\frac12 < z\leq 0,\\
            \sqrt{1-2z}, & \text{if } 0< z< \frac12.
        \end{cases}
    \end{align*}
    
\end{proof}

\subsection{Subdivision-based activation functions} \label{sec_subdivision}

The vast literature on linear subdivision schemes provides a wide range of examples of refinable functions, like the B-Splines, the Daubechies wavelets, etc. We refer to \cite{Dyn92,DL02} for a revision of this theory. In the following, some subdivision concepts will be introduced, marked with italics, without going into details. 

Given some initial data, $f^0$, a \emph{subdivision scheme} iteratively generates a sequence of data, $\{f^k\}_{k\in\N}$, that converges to a function, in some sense. For instance, in the context of geometric modeling, the initial data could be a set of points in the plane, and the scheme generates denser and denser sets of points. In the limit, a smooth curve is obtained, with a shape determined by the initial data.

More in details, a subdivision scheme may consist in the repeated application of the recursive formula
\begin{align} \label{eq_subdivision}
    f^{k+1}_{i} = \sum_{j\in\Z} a_{i-\arity j} f^k_j, \qquad \forall i\in\Z,
\end{align}
where $f^k = [f^k_i]_{i\in\Z}\in \ell_\infty(\Z)$ is a bi-infinite bounded sequence and $a = [a_j]_{j\in\Z}$ is a finitely supported sequence, called the \emph{mask} of the scheme. For a given initial data $f^0$, this iterative process \emph{converges} when the elements of $f^k$ converges uniformly to the point-evaluations of some continuous function. It is well-known in the subdivision theory that the scheme converges if, and only if, there exists a continuous compactly-supported function $\phi:\R\to\R$ (called \emph{basic limit function}) such that
\begin{equation} \label{eq_phi_refinable_general}
    \phi(t) = \sum_{l\in\Z} a_l \phi(\arity t - l), \quad \forall t\in\R.    
\end{equation}
Thus, $\phi$ is a refinable function.
In addition, for the initial data $f^0$, the subdivision scheme converges to the function $\sum_{l\in\Z} f^0_l \phi(t-l)$. To simplify the notation, we assume that the support of $a$ is $[0,\bd+1]$, for some $\bd\in\N$, and it can be proven that the support of $\phi$ is contained in $(0,\bd+1)$ (see Section 2.3 of \cite{DL02}). This assumption is aligned with the B-Spline case of Section \ref{sec_bsplines} and is mild. 
Then, the sum in \eqref{eq_phi_refinable_general} is finite,
\begin{equation} \label{eq_refinable_phi}
    \phi(t) = \sum_{l=0}^{\bd} a_l \phi(\arity t - l), \quad \forall t\in\R,
\end{equation}
which coincides with Definition \ref{def_refinable} with $A = \bd+1$ and $\tau = 0$.

The theory on linear subdivision schemes is well-developed and we can use many of its results. A very convenient theoretical tool for our purposes are the Laurent series. Given a bounded sequence $f = [f_i]_{i\in\Z}$, the associated Laurent series is
$$\hat{f}(z) := \sum_{i\in\Z} f_i z^{i}.$$
With this concept, \eqref{eq_subdivision} can be equivalently written as
$ \hat{f}^{k+1}(z) = \hat{a}(z) \hat{f}^k(z^\arity),$ and a necessary condition for the convergence of the scheme is that $\hat{a}(1) = 2$ and $\hat a(-1) = 0$, or equivalently \(\sum_{i\in\Z} a_{\arity i} = \sum_{i\in\Z} a_{\arity i + 1} = 1\). This implies that the so-called \emph{reproduction of constant functions},
\begin{equation}\label{eq_reproduce_constant}
    \sum_{l\in\Z} \phi(t-l) = 1, \quad \forall t\in\R,
\end{equation}
and that $\hat a(z)$ can be factorized as
\begin{equation} \label{eq_factorization}
    \hat{a}(z) = (1+z)\hat{b}(z),
\end{equation}
for some \emph{Laurent polynomial} $\hat{b}(z)$ (i.e., the sequence $b = [b_i]_{i\in\Z}$ is finitely supported), with $\hat b(1)=1$.

Another property that is essential to obtain a non-decreasing activation function (as asserted in Theorem \ref{thm_sigma_properties}-\ref{thm_sigma_properties:refinable}) is the \emph{monotonicity} of the subdivision scheme. This property ensures that if $f^0$ is non-decreasing, then the limit function is non-decreasing, which recall that it is $\sum_{l\in\Z} f^0_l \phi(t-l)$. A sufficient condition for the monotonicity of the scheme is that $\hat a(z)$ can be factorized as in \eqref{eq_factorization} with $b_l \geq 0$, for all $l\in\Z$.

The last subdivision property that we will use is related with the \emph{generation} of the first degree polynomials. We demand to the basic limit function $\phi$ to fulfill
\begin{equation}\label{eq_reproduce_polynomials}
    t - \frac{\bd+1}2 = \sum_{l\in\Z} l \phi(t-l), \quad \forall t\in\R.
\end{equation}

Before addressing the main result, we need a technical lemma.
\begin{lem} \label{lem:tecnico}
    Let $a,b$ be finitely supported sequences in $\ell_\infty(\Z)$ and let $c\in \ell_\infty(\Z)$ be. Then,
    \[
        \hat a(z) \hat c(z^\arity ) = \sum_{i\in\Z}\left(\sum_{m\in\Z} c_m a_{i-\arity m}\right) z^{i},
    \]
    and, if $\hat a(z) \hat c(z^\arity) = \hat b(z)\hat c(z)$, then
    \(
        \sum_{m\in\Z} c_m a_{i-\arity m} = \sum_{m\in\Z} b_m c_{i-m}, \ \forall i\in\Z.
    \)
\end{lem}
\begin{proof}
    For the first part, we compute:
    \begin{align*}
        \hat a(z) \hat c(z^\arity ) &= \left(\sum_{m\in\Z} a_m z^m \right) \left(\sum_{m\in\Z} c_m z^{\arity m} \right)
        = \left(\sum_{m\in\Z} a_{\arity m} z^{\arity m}+z\sum_{m\in\Z} a_{\arity m + 1} z^{\arity m} \right) \left(\sum_{m\in\Z} c_m z^{\arity m} \right)
        \\
        &= \sum_{i\in\Z}\left(\sum_{m\in\Z} c_m a_{\arity (i-m)}\right) z^{\arity i}+z\sum_{i\in\Z}\left(\sum_{m\in\Z} c_m a_{\arity (i-m)+1}\right) z^{\arity i}\\
        &= \sum_{i\in\Z}\sum_{m\in\Z} c_m a_{\arity i-\arity m} z^{\arity i}+c_m a_{\arity i+1- \arity m} z^{\arity i+1}
        = \sum_{m\in\Z} c_m \sum_{i\in\Z}\left(a_{\arity i-\arity m} z^{\arity i}+a_{\arity i+1- \arity m} z^{\arity i+1}\right)\\
        &= \sum_{m\in\Z} c_m \sum_{i\in\Z}a_{i-\arity m} z^{i}
        = \sum_{i\in\Z}\left(\sum_{m\in\Z} c_m a_{i-\arity m}\right) z^{i}.
    \end{align*}
    The second part follows from comparing term by term the last series with the series
    \[
        \hat b(z)\hat c(z) = \sum_{i\in\Z}\left(\sum_{m\in\Z} b_m c_{i-m}\right) z^i.
    \]
\end{proof}

For a given basic limit function $\phi$, we can define the associated activation function $\sigma$ as in Definition \ref{defi_sigma_from_phi}. The following theorem states some properties of $\sigma$ that are inherited from $\phi$.

\begin{thm} \label{thm_sigma_properties}
    Let $\phi:\R\to\R$ be the basic limit function, whose support is contained in $(0,\bd+1)$, associated to a convergent subdivision scheme with mask $a$. Let $b\in\ell_\infty(\Z)$ be finitely supported, such that $\hat{a}(z) = (1+z)\hat{b}(z)$.
    Then, the associated activation function $\sigma:\R\to\R$ fulfills:
    \begin{enumerate}[label=(\alph*)]
        \item \label{thm_sigma_properties:suport} $\sigma(t) = -\frac12$, if $t\leq -\frac{\bd}2$, $\sigma(t) = \frac12$, if $t\geq\frac{\bd}2$, and $\sigma(t) = -\frac12 + \sum_{m=0}^{d-1} \phi(t+\frac{\bd}2-m)$, if $-\frac{\bd}2\leq t\leq \frac{\bd}2$.
        \item \label{thm_sigma_properties:smooth} $\sigma$ is as smooth as $\phi$.
        \item \label{thm_sigma_properties:difference} $\sigma(t) - \sigma(t-1) = \phi\left(t+\frac{\bd}2\right)$.
        \item \label{thm_sigma_properties:symmetric} If $\phi(t) = \phi(d+1-t)$, $\forall t\in\R$,  then $\sigma(-t) = -\sigma(t)$, $\forall t\in\R$.
        \item \label{thm_sigma_properties:refinable} $\sigma$ is refinable with $A^\sigma := \bd+1$, $a^\sigma := b$ and $\tau^\sigma := \frac{d}2$. In addition, if $b_l\geq 0$, $\forall l\in\Z$, then $\sigma$ is non-decreasing.
        \item \label{thm_sigma_properties:sum} If $\sum_{i\in\Z} i \phi(t-i) = t - \frac{\bd+1}2$, $\forall t\in \R$, and $B\geq d$ is chosen, then $\sigma$ sums the identity in the interval $I = [-\frac{B-d+1}2,\frac{B-d+1}2]$ with $\mu = \frac{B-1}2$.
    \end{enumerate}
\end{thm}
\begin{proof}
    
    \ref{thm_sigma_properties:suport} For $t\leq -\bd/2$, we have that $t+\frac{\bd}2 -m \leq -m\leq 0$, for all $m\geq 0$. Thus $\phi(t+\frac{\bd}2 -m) = 0$ and $\sigma(t) = -\frac12$ since the sum in \eqref{eq_sigma_from_phi} vanishes. For $t\geq \bd/2$, $t+\frac{\bd}2 -m \geq \bd-m \geq \bd+1$ for all $m\leq -1$. Then, $\phi(t+\frac{\bd}2 -m) = 0$ for all $m\leq -1$ and we can extend the sum range:
    \[
    \sigma(t) = -\frac12 + \sum_{m\in\Z} \phi\left(t+\frac{\bd}2 -m\right) \overset{\eqref{eq_reproduce_constant}}{=} -\frac12 + 1 = \frac12.
    \]
    For $-\bd/2\leq t\leq \bd/2$, we have that $t+\frac{\bd}2 -m \leq d-m \leq 0$ if $m\geq d$. Thus, $\phi(t+\frac{\bd}2 -m)=0$ and the sum in \eqref{eq_sigma_from_phi} is finite:
    \[
    \sigma(t) = -\frac12 + \sum_{m=0}^{d-1} \phi\left(t+\frac{\bd}2 -m\right).
    \]
    
    \ref{thm_sigma_properties:smooth} Since $\phi$ is compactly supported, the sum in \eqref{eq_sigma_from_phi} is finite for each $t\in\R$. Thus, $\sigma$ is as smooth as $\phi$.
    For \ref{thm_sigma_properties:difference}, we have that
    \begin{align*}
        \sigma(t) - \sigma(t-1) &= \sum_{m=0}^\infty \phi\left(t+\frac{\bd}2-m\right) - \sum_{m=0}^\infty \phi\left(t+\frac{\bd}2-m-1\right) \\
        &= \sum_{m=0}^\infty \phi\left(t+\frac{\bd}2-m\right) - \sum_{m=1}^\infty \phi\left(t+\frac{\bd}2-m\right) = \phi\left(t+\frac{\bd}2\right).
    \end{align*}
    
    \ref{thm_sigma_properties:symmetric} Suppose that $\phi(t) = \phi(d+1-t)$, $\forall t\in\R$. We will prove by induction that $\sigma(t) = -\sigma(-t)$, if $|t|\geq (\bd-n)/2$, for $n=0,1,\ldots,\bd$. Observe the case $n=d$ prove the symmetry for any $t\in\R$. The case $n=0$ is a consequence of \ref{thm_sigma_properties:suport}. Now, suppose that $\sigma(t) = -\sigma(-t)$, if $|t|\geq (\bd-(n-1))/2$. Let us consider $t\geq (\bd-n)/2$ (the case $t\leq -(\bd-n)/2$ is analogous). Then $t+1\geq (\bd-(n-1))/2$ and
    \begin{align*}
        \sigma(t) &\overset{\ref{thm_sigma_properties:difference}}{=} \sigma(t+1) - \phi\left(t+1+\frac{\bd}2\right) \overset{\text{I.H.}}= -\sigma(-t-1) - \phi\left(t+1+\frac{\bd}2\right)\\
        &\overset{\phi \text{ sym.}}{=} -\sigma(-t-1) - \phi\left(d+1-\left(t+1+\frac{\bd}2\right)\right)
        = -\left(\sigma(-t-1) + \phi\left(-t+\frac{\bd}2\right)\right)
        \overset{\ref{thm_sigma_properties:difference}}{=} -\sigma(-t).
    \end{align*}
    
    For \ref{thm_sigma_properties:refinable}, we perform a deductive proof to show if other kinds of refinable activation functions $\sigma$ are feasible. Denote by $a^\phi$ the mask associated to $\phi$. We construct a $\bar \sigma$ that is a linear combination of $\phi$ and its shifts. For some $c\in\ell_\infty(\Z)$ and $c',\tau\in\R$, we consider
    \begin{equation} \label{eq_sigma}
        \bar \sigma(t) = c' + \sum_{m\in\Z} c_m \phi(t+\tau-m),
    \end{equation}
    and suppose that there exists some finitely supported $b\in\ell_\infty(\Z)$ such that $\hat{a}^\phi(z) \hat c(z^\arity)=  \hat{b}(z) \hat c(z)$. Then,
    \begin{align*}
        \bar\sigma(t-\tau)-c' &=
        \sum_{m\in\Z} c_m \phi(t-m)
        \overset{\phi\text{ refi.}}{=} \sum_{m\in\Z} c_m \sum_{i\in\Z} a^\phi_i\phi(\arity(t-m)-i)
        =\sum_{m\in\Z} c_m \sum_{i\in\Z} a^\phi_{i-\arity m}\phi(\arity t-i)  \\
        &=\sum_{i\in\Z} \phi(\arity t-i)\sum_{m\in\Z} c_m a^\phi_{i-\arity m} \overset{\text{Lemma \ref{lem:tecnico}}}= \sum_{i\in\Z} \phi(\arity t-i)\sum_{m\in\Z} b_m c_{i-m}
        = \sum_{m\in\Z} b_m\sum_{i\in\Z} c_{i-m} \phi(\arity t-i)\\
        &= \sum_{m\in\Z} b_m\sum_{i\in\Z} c_{i} \phi(\arity t-i-m).
    \end{align*}
    If $\sum_{m\in\Z} b_m = \hat b(1) = 1$, then
    \begin{align*}
        \bar \sigma(t-\tau) &= c' + \sum_{m\in\Z} b_m\sum_{i\in\Z} c_i \phi(\arity t-m-i) = \sum_{m\in\Z} b_mc' + \sum_{m\in\Z} b_m\sum_{i\in\Z} c_i \phi(\arity t-m-i)\\
        &= \sum_{m\in\Z} b_m\left(c' + \sum_{i\in\Z} c_i \phi(\arity t-m-i)\right), \\
        \bar \sigma(t) &= \sum_{m\in\Z} b_m\left(c' + \sum_{i\in\Z} c_i \phi(\arity t+2\tau-m-i)\right) \overset{\eqref{eq_reproduce_constant}}{=} \sum_{m\in\Z} b_m\bar \sigma(2t+\tau-m).
    \end{align*}
    We conclude that $\bar \sigma$ defined in \eqref{eq_sigma} is refinable provided that $ \hat{a}(z) \hat c(z^\arity)=  \hat{b}(z) \hat c(z)$ with $\hat b(1) = 1$.
    In particular, the activation function of Definition \ref{defi_sigma_from_phi} fulfils this with $c'=-\frac12, \tau = \frac{d}2$ and $c_m = 1$, if $m\geq 0$, and $c_m = 0$, if $m<0$.
    Observe that $\hat c(z) = \sum_{m=0}^\infty z^m = (1-z)^{-1}$ and, thus, the requirement $ \hat{a}(z) \hat c(z^\arity)=  \hat{b}(z) \hat c(z)$ is equivalent to $\hat{a}(z) = (1+z) \hat{b}(z)$, which is true since the scheme is convergent by hypothesis, which also implies that $\hat b(1)=1$. Since the support of $a$ is $[0,\bd+1]$, then the support of $b$ is $[0,\bd]$. Summarizing, we deduced that
\[
    \sigma(t) = \sum_{l=0}^{\bd} b_l\sigma\left(\arity t+ \frac{d}2-l\right),
\]
    i.e., $\sigma$ is refinable with $A^\sigma = d+1$, $a^\sigma = b$ and $\tau^\sigma = \frac{d}2$. 
    
    In addition, $b$ are the coefficients of what is commonly referred to as the \emph{difference scheme}. It is well-established that if $a^\sigma_{l} = b_l \geq 0$, then $\sum_{m \in \mathbb{Z}} c_m \phi(t - m)$ is a non-decreasing function for any non-decreasing sequence $c$. Specifically, since our chosen sequence $c$ is non-decreasing, it follows that $\sigma$ is also non-decreasing.
    
    For \ref{thm_sigma_properties:sum}, suppose that $t\in [-\frac{B-d+1}2,\frac{B-d+1}2]$.
    \begin{align*}
        \sum_{l=0}^{B-1} \sigma\left(t + \frac{B-1}2 -l\right)
        &= \sum_{l=0}^{B-1} \left(-\frac12 +\sum_{m=0}^\infty \phi\left(t + \frac{B-1}2 + \frac{d}2 -l -m\right)\right) \\
        &= -\frac{B}2 + \sum_{l=0}^{B-1} \sum_{m=0}^\infty \phi\left(t + \frac{B+d-1}2 -l -m\right).
    \end{align*}
    Since the support of $\phi$ is contained in $(0,d+1)$, and $t + \frac{B+d-1}2 -l -m \leq \frac{B-d+1}2 + \frac{B+d-1}2  -l -m = B -l -m \leq 0$ if, and only if, $B-l\leq m$, we have that
    \begin{align*}
        \frac{B}2 + \sum_{l=0}^{B-1} \sigma\left(t + \frac{B-1}2 -l\right) &= \sum_{l=0}^{B-1} \sum_{m=0}^{B-l-1} \phi\left(t + \frac{B+d-1}2 -l -m\right) = \sum_{l=0}^{B-1} \sum_{m=l}^{B-1} \phi\left(t + \frac{B+d-1}2 -m\right) \\
        =& \sum_{m=0}^{B-1} \sum_{l=0}^m \phi\left(t + \frac{B+d-1}2 -m\right) = \sum_{m=0}^{B-1} \phi\left(t + \frac{B+d-1}2 -m\right)\sum_{l=0}^m  1\\
        =& \sum_{m=0}^{B-1} (m+1)\phi\left(t + \frac{B+d-1}2 -m\right).
    \end{align*}
    Observe that the term $m=-1$ can be added, since it is zero. Thus,
    \begin{align*}
        \frac{B}2 + \sum_{l=0}^{B-1} \sigma\left(t + \frac{B-1}2 -l\right) &= \sum_{m=-1}^{B-1} (m+1)\phi\left(t + \frac{B+d-1}2 -m\right) 
        = \sum_{m=0}^{B} m\phi\left(t + \frac{B+d+1}2 -m\right).
    \end{align*}
    In addition, the sum can be extended to all the integers:
    \begin{align*}
        m\leq -1,& \quad \rightarrow \quad t + \frac{B+d+1}2 -m \geq t + \frac{B+d+1}2 +1 \geq -\frac{B-d+1}2 + \frac{B+d+1}2 +1 = d+1.\\
        m \geq B+1,& \quad \rightarrow \quad t + \frac{B+d+1}2 -m \leq t + \frac{B+d+1}2 -B-1 \leq \frac{B-d+1}2  + \frac{B+d+1}2 -B-1 = 0.
    \end{align*}
    In any of both cases, $\phi(t + \frac{B+d+1}2 -m) = 0$, due to the support of $\phi$ is contained in $(0,d+1)$. Finally, 
    using the hypothesis we conclude that
    \begin{align*}
        \sum_{l=0}^{B-1} \sigma\left(t + \frac{B-1}2 -l\right) &= -\frac{B}2 + \sum_{m\in\Z} m\phi\left(t + \frac{B+d+1}2 -m\right) = -\frac{B}2 + t + \frac{B+d+1}2 - \frac{d+1}2 = t.
    \end{align*}
\end{proof}

\begin{rmk}
    As a consequence of Theorem \ref{thm_sigma_properties}-\ref{thm_sigma_properties:refinable}, we deduce that not all basic limit functions $\phi$ are appropriate for constructing refinable activation functions. In particular, ensuring that $b_l\geq 0$, $\forall l\in\Z$, holds is essential for obtaining a non-decreasing activation function. In subdivision theory, schemes that satisfy $b_l\geq 0$, $\forall l\in\Z$, are termed monotone. Such schemes inherently fulfil a variation diminishing property, which suppresses oscillatory behavior (see \cite{Yadshalom93}). This is especially important since oscillatory activation functions may lead to performance issues.

    In addition to the B-Spline case studied in this paper, other examples of monotone subdivision schemes that generates first-degree polynomials---and thereby allow the definition of proper activation functions as described---include those based on weighted polynomial regression (see \cite{LY24}).
\end{rmk}

\begin{rmk}
    Only a few subdivision schemes yield basic limit functions $\phi$ that admit a closed-form expression, such as B-Splines and exponential B-Splines schemes (see \cite{Dyn92,UB05}).
    
    Nevertheless, $\sigma$ can be approximated with arbitrary precision using subdivision. We recall that, when a convergent subdivision scheme is recursively applied to an initial sequence $f^0$, the process converges to $\sum_{l\in\Z} f^0_l \phi(t-l)$. Hence, by considering $f^0_l := -\frac12$, if $l<0$, and $f^0_l := \frac12$, if $l\geq 0$, we obtain the limit function
    \begin{align*}
        -\frac12 \sum_{l<0} \phi(t-l) + \frac12 \sum_{l\geq0} \phi(t-l) \overset{\eqref{eq_reproduce_constant}}{=}& -\frac12 \sum_{l<0} \phi(t-l) + \frac12 \sum_{l\geq0} \phi(t-l) - \frac12 + \frac12 \sum_{l\in\Z} \phi(t-l) \\
        =& -\frac12 + \sum_{l\geq 0} \phi(t-l) = \sigma\left(t-\frac{d}2\right).
    \end{align*}
    In conclusion, $\sigma\left(t-\frac{d}2\right)$ is the limit function of the subdivision scheme applied to such initial sequence $f^0$.
\end{rmk}

\begin{proof}[Proof of Theorem \ref{thm_sigmaB_properties}]
    For \ref{thm_sigmaB_properties_support}-\ref{thm_sigmaB_properties_sum}, we apply Theorem \ref{thm_sigma_properties}. All the requirements are met since $\phi_{B^\bd}$ is $\cC^{d-1}$, fulfills $\phi(t) = \phi(d+1-t)$, fulfils \eqref{eq_bsplines_identity} and is the basic limit function of the subdivision scheme with mask $a^{\phi_{B^\bd}}_l = 2^{-\bd}\binom{\bd+1}{l}$, for $l=0,\ldots,\bd+1$, and $a^{\phi_{B^\bd}}_l=0$ otherwise. The associated Laurent polynomial is $\hat a^{\phi_{B^\bd}}(z) = 2^{-\bd} (1+z)^{d+1}$, thus $\hat b^{\phi_{B^\bd}}(z) = 2^{-\bd} (1+z)^{d}$ and then $a_l = b^{\phi_{B^\bd}}_l = 2^{-\bd}\binom{\bd}{l}$, for $l=0,\ldots,\bd$.
    
    For \ref{thm_sigmaB_properties_expression}, we will prove that the identity is true for any $t\in [-d/2-1+n,-d/2+n]$, by induction on $n$. By \ref{thm_sigmaB_properties_support}, $\sigma_{B^\bd}(t) = -1/2$ for $t\leq -\bd/2$, then this is fulfilled for any $n\leq 0$, since
    $$0 \leq \max\{t+d/2-l,0\} \leq \max\{-d/2+n+d/2-l,0\} \leq \max\{-l,0\} = 0, \qquad l\geq 0.$$
    Now, we assume that it is true for $n-1$ and we prove it for $n$. Using \ref{thm_sigma_properties:difference} of Theorem \ref{thm_sigma_properties}, we have that
    \begin{align*}
        \sigma_{B^\bd}(t) =& \sigma_{B^\bd}(t-1) + \phi_{B^{\bd}}(t+\frac12) \\
        \overset{\text{\eqref{eq_bsplines_direct} \& I.H.}}{=}& -\frac12 + \sum_{l=0}^{\bd} \frac{(-1)^l}{\bd!} \binom{\bd}{l} \max\{t-1+d/2-l,0\}^\bd + \sum_{l=0}^{d+1} \frac{(-1)^{l}}{d!}\binom{d+1}{l} \max\{t+d/2-l,0\}^{d}.
    \end{align*}
    Performing the summation change $l\to l-1$ in the first sum and splitting $\binom{\bd+1}{l} = \binom{\bd}{l-1} + \binom{\bd}{l}$ in the second sum (which demands considering $\binom{\bd}{-1} = 0 = \binom{\bd}{\bd+1}$), we obtain
    \begin{align*}
        \sigma_{B^\bd}(t) =& -\frac12 + \sum_{l=1}^{\bd+1} \frac{(-1)^{l-1}}{\bd!} \binom{\bd}{l-1} \max\{t+d/2-l,0\}^\bd + \sum_{l=0}^{d+1} \frac{(-1)^{l}}{d!}\binom{d}{l-1} \max\{t+d/2-l,0\}^{d} \\
        &+ \sum_{l=0}^{d+1} \frac{(-1)^{l}}{d!}\binom{d}{l} \max\{t+d/2-l,0\}^{d}.
    \end{align*}
    Observe that the term $l=0$ vanishes in the second sum, and the term $l=d+1$ vanishes in the third sum. Consequently, the first and second sums cancel out, leading to:
    \begin{align*}
        \sigma_{B^\bd}(t) =& -\frac12 + \sum_{l=0}^{d} \frac{(-1)^{l}}{d!}\binom{d}{l} \max\{t+d/2-l,0\}^{d}.
    \end{align*}
    
    For \ref{thm_sigmaB_properties_recursive}, denote by $T_l$ the shift operator $(T_l \phi)(t) = \phi(t-l)$. By definition,
    \(
    \sigma_{B^\bd} = -\frac12 + \sum_{m=0}^\infty T_{m-d/2} \phi_{B^\bd}.
    \)
    Since $\phi$ is compactly supported, for a given $t\in\R$, there exists some $M\in\N$ such that
    \begin{align*}
        \sigma_{B^\bd}(t) &= -\frac12 + \sum_{m=0}^M (T_{m-d/2} \phi_{B^\bd})(t), \quad
        \sigma_{B^{\bd-1}}(t) = -\frac12 + \sum_{m=0}^M (T_{m-(d-1)/2} \phi_{B^{\bd-1}})(t).
    \end{align*}
    By \eqref{eq_bsplines},
    \begin{align*}
        \sigma_{B^\bd}(t) &= -\frac12 + \sum_{m=0}^M (T_{m-d/2} (\phi_{B^{d-1}}*\phi_{B^0}))(t) = -\frac12 + \sum_{m=0}^M (T_{m-d/2+1/2}T_{-1/2} (\phi_{B^{d-1}}*\phi_{B^0}))(t).
    \end{align*}
    Using the translational equivarence of the convolution, i.e. $T_l (\phi_1*\phi_2) = (T_l\phi_1)*\phi_2=\phi_1*(T_l \phi_2)$, we obtain that
    \begin{align*}
        \sigma_{B^\bd}(t) &= -\frac12 + \sum_{m=0}^M ((T_{m-d/2+1/2}\phi_{B^{d-1}})*(T_{-1/2}\phi_{B^0}))(t)\\
        &= -\frac12 + \sum_{m=0}^M ((T_{m-(d-1)/2}\phi_{B^{d-1}})*U)(t)
        = -\frac12 + \left(U*\sum_{m=0}^M T_{m-(d-1)/2}\phi_{B^{d-1}}\right)(t)\\
        &= \left(U*\left(-\frac12 + \sum_{m=0}^M T_{m-(d-1)/2}\phi_{B^{d-1}}\right)\right)(t)= \left(U* \sigma_{B^{d-1}}\right)(t).
    \end{align*}
    
    Finally, \ref{thm_sigmaB_properties_derivative} is a consequence of \ref{thm_sigmaB_properties_recursive} and Theorem \ref{thm_sigma_properties}-\ref{thm_sigma_properties:difference}, since
    \[
    \sigma_{B^\bd}'(t) = \sigma_{B^{\bd-1}}(t+1/2) - \sigma_{B^{\bd-1}}(t-1/2) = \phi_{B^{\bd-1}}\left(t+\frac12+\frac{d-1}2\right)= \phi_{B^{\bd-1}}\left(t+\frac{d}2\right).
    \]
\end{proof}

\section{Conclusions} \label{sec_conclusions}

In this work, a new class of activation functions was introduced, characterized by two key properties: they are refinable, and they sum the identity.

It was demonstrated how these properties can be utilized to add new neurons and layers without altering the NN output, at least for any input data belonging to a specific, arbitrarily large, set. Explicit formulas for constructing the new neurons and layers were provided and pseudocode for the algorithms was presented.

The proposed theoretical results were applied to define the spline activation functions within this class. Specific properties were presented for these functions, including formulas for computing $\sigma_{B^1}$ and $\sigma_{B^2}$ as well as their derivatives in a manner compatible with backpropagation.

Several open questions remain for future investigation. We hope that these results can lead to improvements when combined with multi-level training algorithms, such as those employed in structural learning, by either enhancing NN performance or accelerating convergence during the training phase, similarly to other function-preserving transformations (e.g., \cite{CGS15,WWRC16,WWCW19}). In particular, we wonder whether this potential improvement depends on the \emph{approximation order} of the subdivision scheme, linked to the refinable activation function. Finally, an efficient and accurate practical implementation requires closed-form expressions for $\sigma$ and $\sigma'$. In this line, extending Proposition \ref{prop_backpropagation} to the case $d > 2$ or to more general refinable functions would be highly beneficial.

\section*{Acknowledgments}
This research has been supported through projects CIAICO/2021/227, PID2020-117211GB-I00 funded by MCIN/AEI/10.13039/501100011033 and PID2023-146836NB-I00 funded by MCIU/AEI/10.13039/5011 00011033.

\section*{Reproducibility}

The Wolfram Mathematica notebook file containing symbolic computations that complement the proofs presented in the manuscript, are available on Github:

\noindent\url{https://github.com/serlou/refinable-neural-networks}.

\section*{Declaration of Generative AI and AI-assisted technologies in the writing process}

During the preparation of this work the author used ChatGPT in order to improve readability and language. After using this service, the author reviewed and edited the content as needed and take full responsibility for the content of the publication.

\bibliography{biblio}

\appendix

\section{Algorithms for refining neural networks}
\label{algorithms}

Building upon the theoretical results established in Section \ref{sec_refining}, this appendix presents three algorithms: one for widening network layers via neuron splitting and two for inserting new layers.

To ensure notational consistency across all algorithms and the results of Section \ref{sec_refining}, let $\mathcal{N} = \{\hat L^j:\R^{\hat n_j}\to\R^{\hat n_{j+1}}\}_{j=0}^{m_L-1}$ denote a NN with $m_L$ layer operators. Within an object-oriented programming framework, we assume the availability of getter and setter methods to access the following properties:
\begin{itemize}
    \setlength\itemsep{0em}
    \item Parameters of each layer operator: Activation function, weights and biases.
    \item Parameters $a, \nA, \tau$ defining the refinability of a function (Definition \ref{def_refinable}).
    \item Parameters $\mu, B, I$ for functions that sum the identity (Definition \ref{defi_sum_identity}), along with $\delta$, defined as the largest positive value satisfying $(-\delta,\delta)\subset I$.
\end{itemize}

Algorithm \ref{alg_grow_layer}, derived from Theorem \ref{thm_increase_neurons}, enables neuron splitting to increase the width of a selected layer. Algorithms \ref{alg_insert_layer_pre} and \ref{alg_insert_layer_post} allow inserting a new layer, but they differ in the number of neurons allocated to the new layer, which is determined by the size of the preceding or subsequent layer, respectively. They draw upon Theorem \ref{thm_insert_layer_pre} with Remark \ref{rmk_choose_beta_1} and Theorem \ref{thm_insert_layer_post} with Remark \ref{rmk_choose_beta_2}, respectively.

\begin{algorithm}[H]
    \caption{Splitting Neurons in an Existing Layer \label{alg_grow_layer}}
    \begin{algorithmic}[1]
        \Require Target layer index $j_*$, neuron subset $\Omega \subset \{1,\ldots,n_{j_*}\}$ to split, and network $\mathcal{N}$ whose $\hat L^{j_*-1}$ has a refinable activation function.
        \Ensure Updated network $\mathcal{N}$ with modified layers $\hat L^{j_*-1},\hat L^{j_*}$.
        \Statex
        
        \Statex \textbf{Align notation with Section \ref{sec_increase_neurons}:}
        \State \(L^0 \gets \hat L^{j_*-1}\), \(L^1 \gets \hat L^{j_*}\)
        \State \(n_0 \gets \hat n_{j_*-1}\), \(n_1 \gets \hat n_{j_*}\), \(n_2 \gets \hat n_{j_*+1}\)
        
        \Statex
        \Statex \textbf{Initialize parameters:}
        \State Get \(W^0, b^0, \sigma^0\) from \(L^0\)
        \State Get \(W^1\) from \(L^1\)
        \State \(m \gets |I|\) \Comment{Cardinality of interval \(I\)}
        \State Get \(a, \nA, \tau\) from \(\sigma^0\)
        \State \(\bar n_1 \gets n_1 + m(\nA - 1)\)
        \State Initialize matrices \(\overline{W}^0 \in \R^{\bar n_1 \times n_0}\), \(\overline{b}^0 \in \R^{\bar n_1}\), \(\overline{W}^1 \in \R^{n_2 \times \bar n_1}\)
        
        \Statex
        \Statex \textbf{Apply Theorem \ref{thm_increase_neurons}:}
        \State \(k \gets 0\) \Comment{Index for new matrices}
        \For{\(i = 0\) \textbf{to} \(n_1 - 1\)} \Comment{Iterate over original neurons}
        \If{\(i \in \Omega\)}
        \For{\(l = 0\) \textbf{to} \(\nA - 1\)} \Comment{Split neuron \(i\) into \(\nA\) subunits}
        \State \(\overline{W}^0_{k,:} \gets \arity W^0_{i,:}\)
        \State \(\overline{b}^0_k \gets \arity b^0_i + \tau - l\)
        \State \(\overline{W}^1_{:,k} \gets a_{l} W^1_{:,i}\)
        \State \(k \gets k + 1\)
        \EndFor
        \Else \Comment{Retain original neuron}
        \State \(\overline{W}^0_{k,:} \gets W^0_{i,:}\) 
        \State \(\overline{b}^0_k \gets b^0_i\)
        \State \(\overline{W}^1_{:,k} \gets W^1_{:,i}\)
        \State \(k \gets k + 1\)
        \EndIf
        \EndFor
        \State Set \(\overline{W}^0, \overline{b}^0\) in \(\hat L^{j_*-1}\)
        \State Set \(\overline{W}^1\) in \(\hat L^{j_*}\)
    \end{algorithmic}
\end{algorithm}

\begin{algorithm}[!h]
    \caption{Inserting a New Layer (Option 1) \label{alg_insert_layer_pre}}
    \begin{algorithmic}[1]
        \Require $\mathcal{N}$ a NN, $j_*$ the insertion position, $\sigma^0$ an activation function summing the identity, and training data $\{y_k\}_{k=0}^{m_T-1}\subset\R^{n_0}$.
        \Ensure Updated network $\mathcal{N}$ with an extra layer at $j_*$.
        \Statex

        \Statex \textbf{Align notation with Section \ref{sec_insert_layer}:}
        \State $L \gets \hat L^{j_*-1}$
        \State \(n_0 \gets \hat n_{j_*-1}\), \(n_1 \gets \hat n_{j_*}\)

        \Statex
        \Statex \textbf{Initialize parameters:}
        \State Get \(W, b, \sigma\) from \(L\)
        \State Get \(\mu, B, \delta\) from \(\sigma^0\)
        \State $\bar n \gets B n_0$ \Comment{New layer size}
        \State Initialize $W^0 \gets 0 \in\R^{\bar n\times n_0}, b^0 \gets 0 \in\R^{\bar n}, W^1 \gets 0\in\R^{n_1\times \bar n}$

        \Statex
        \Statex \textbf{Compute $\beta$ as in Remark \ref{rmk_choose_beta_1}:}
        \For{$j=0$ to $j_*-2$}
        \For{$k=0$ to $m_T-1$}
        \State $y_k \gets \hat L^j(y_k)$ \Comment{Compute the output of the previous layers}
        \EndFor
        \EndFor
        \State $\beta \gets \frac{\delta}{2\max_{k} \{\|y_k\|_\infty\}}$

        \Statex
        \Statex \textbf{Apply Theorem \ref{thm_insert_layer_pre}:}
        \For {$i=0$ to $n_0-1$}
        \For{$l=0$ to $B-1$}
        \State $W^0_{l+iB,i} \gets \beta$
        \State $b^0_{l+iB} \gets \mu - l$
        \State $W^1_{:,l+iB} \gets \frac1\beta W_{:,i}$
        \EndFor
        \EndFor

        \Statex
        \Statex \textbf{Update network:}
        \State $L^0(x) := \sigma^0(W^0 x + b^0)$
        \State $L^1(x) := \sigma(W^1 x + b)$
        \State Set \(\hat L^0,\ldots,\hat L^{j_*-2},L^0,L^1,\hat L^{j_*},\ldots,\hat L^{m_L-1}\) as the sequence of layer operators in $\mathcal{N}$
    \end{algorithmic}
\end{algorithm}

\begin{algorithm}[!h]
    \caption{Inserting a New Layer (Option 2) \label{alg_insert_layer_post}}
    \begin{algorithmic}[1]
        \Require $\mathcal{N}$ a NN, $j_*$ the insertion position, $\sigma^0$ an activation function summing the identity, and training data $\{y_k\}_{k=0}^{m_T-1}\subset\R^{n_0}$.
        \Ensure Updated network $\mathcal{N}$ with an extra layer at $j_*$.
        \Statex

        \Statex \textbf{Align notation with Section \ref{sec_insert_layer}:}
        \State $L \gets \hat L^{j_*-1}$
        \State \(n_0 \gets \hat n_{j_*-1}\), \(n_1 \gets \hat n_{j_*}\)

        \Statex
        \Statex \textbf{Initialize parameters:}
        \State Get \(W, b, \sigma\) from \(L\)
        \State Get \(\mu, B, \delta\) from \(\sigma^0\)
        \State $\bar n \gets B n_1$ \Comment{New layer size}
        \State Initialize $W^0 \gets 0 \in\R^{\bar n\times n_0}, b^0 \gets 0 \in\R^{\bar n}, W^1 \gets 0\in\R^{n_1\times \bar n}$

        \Statex
        \Statex \textbf{Compute $\beta$ as in Remark \ref{rmk_choose_beta_2}}
        \For{$j=0$ to $j_*-2$}
        \For{$k=0$ to $m_T-1$}
        \State $y_k \gets \hat L^j(y_k)$ \Comment{Compute the output of the previous layers}
        \EndFor
        \EndFor
        \State $\beta \gets \frac{\delta}{2\max_{k} \|W y_k + b\|_\infty}$
        
        \Statex
        \Statex \textbf{Apply Theorem \ref{thm_insert_layer_post}:}
        \For {$i=0$ to $n_1-1$}
        \For{$l=0$ to $B-1$}
        \State $W^0_{i+l n_1,:} \gets \beta W_{i,:}$
        \State $b^0_{i + l n_1} \gets \beta b_i + \mu - l$
        \State $W^1_{i,i+l n_1} \gets \frac1\beta$
        \EndFor
        \EndFor

        \Statex
        \Statex \textbf{Update network:}
        \State $L^0(x) := \sigma^0(W^0 x + b^0)$
        \State $L^1(x) := \sigma(W^1 x)$
        \State Set \(\hat L^0,\ldots,\hat L^{j_*-2},L^0,L^1,\hat L^{j_*},\ldots,\hat L^{m_L-1}\) as the sequence of layer operators in $\mathcal{N}$
    \end{algorithmic}
\end{algorithm}

\end{document}